\documentclass[3p]{elsarticle}
\usepackage{amssymb}
\usepackage{amsfonts}
\usepackage{tikz}
\usetikzlibrary{decorations.pathmorphing,patterns,arrows,shapes,snakes,automata,backgrounds,petri,calc}
\newtheorem{theorem}{Theorem}%[section]
\newtheorem{conjecture}[theorem]{Conjecture}
\newtheorem{corollary}[theorem]{Corollary}
\newtheorem{lemma}[theorem]{Lemma}
\newtheorem{observation}[theorem]{Observation}

\newtheorem{claim}{Claim}

\newproof{proof}{Proof}

\begin{document}
\begin{frontmatter}

\title{On the total neighbour sum distinguishing index of graphs with bounded maximum average degree}

\author[LaBRI]{H. Hocquard \fnref{FRgrant}}
\ead{herve.hocquard@labri.fr}
%\address{AGH University of Science and Technology, al. A. Mickiewicza 30, 30-059 Krakow, Poland}

\author[agh]{J. Przyby{\l}o\corref{cor1}\fnref{grantJP,MNiSW}} %,fn2
\ead{jakubprz@agh.edu.pl}%, phone: 048-12-617-46-38,  fax: 048-12-617-31-65}

\cortext[cor1]{Corresponding author}
\fntext[FRgrant]{Supported by CNRS-PICS Project no. 6367 ``GraphPar''.}
\fntext[grantJP]{Supported by the National Science Centre, Poland, grant no. 2014/13/B/ST1/01855.}
\fntext[MNiSW]{Partly supported by the Polish Ministry of Science and Higher Education.}

\address[LaBRI]{LaBRI (Universit\'e de Bordeaux), 351 cours de la Lib\'eration, 33405 Talence Cedex, France}
\address[agh]{AGH University of Science and Technology, al. A. Mickiewicza 30, 30-059 Krakow, Poland}

\begin{abstract}
A proper total $k$-colouring of a graph $G=(V,E)$ is an assignment $c : V \cup E\to \{1,2,\ldots,k\}$ of colours to the edges and the vertices of $G$ such that no two adjacent edges or vertices and no edge and its end-vertices are associated with the same colour. A total neighbour sum distinguishing $k$-colouring, or tnsd $k$-colouring for short, is a proper total $k$-colouring such that $\sum_{e\ni u}c(e)+c(u)\neq \sum_{e\ni v}c(e)+c(v)$ for every edge $uv$ of $G$. We denote by $\chi''_{\Sigma}(G)$ the total neighbour sum distinguishing index of $G$, which is the least integer $k$ such that a tnsd edge $k$-colouring of $G$ exists.
It has been conjectured that $\chi''_{\Sigma}(G) \leq \Delta(G) + 3$ for every graph $G$.
In this paper we confirm this conjecture
for any graph $G$ with ${\rm mad}(G)<\frac{14}{3}$ and $\Delta(G) \geq 8$.

\end{abstract}


\end{frontmatter}

\section{Introduction}
A \emph{proper total $k$-colouring} of a graph $G=(V,E)$ is an assignment $c$ of colours from the set $\{1,2,\ldots,k\}$ to the edges and the vertices of $G$ such that
adjacent edges and vertices are coloured differently and the colour of every edge is distinct from those assigned to its end-vertices.
A \emph{total neighbour sum distinguishing $k$-colouring} of $G$, or \emph{tnsd $k$-colouring} for short,
is its proper total $k$-colouring $c$ such that for every edge $uv\in E$, there is no \emph{conflict} between $u$ and $v$,
{\it i.e.}, $s(u)\neq s(v)$, where $s(w)$ is the sum of colours taken on the edges incident with $w$ and the colour of the vertex $w$ for $w\in V$.
In other words, for every vertex $w\in V$, $\displaystyle s(w) = \sum_{e \in E_w} c(e)+c(w)$, where $E_w$ is the set of edges incident with $w$ in $G$.
We denote by $\chi''_{\Sigma}(G)$ the \emph{total neighbour sum distinguishing index} of $G$, which is the least integer $k$ such that a tnsd  $k$-colouring of $G$ exists.
The roots of this branch of graph theory date back to the '80s, and the papers~\cite{ChartrandErdosOellermann,Chartrand}
on degree irregularities in graphs (and multigraphs) and the parameter \emph{irregularity strength} of a graph. For more details concerning a motivation for investigating integer graph colourings and a few crucial results on the irregularity strength see {\it e.g.}~\cite{Aigner,%Bohman_Kravitz,
Lazebnik,%Faudree,
Frieze,KalKarPf,Lehel,MajerskiPrzybylo2,Nierhoff}.

By definition and the requirement of properness of colourings investigated, the total neighbour sum distinguishing index of every graph $G$ is not smaller than $\Delta(G)+1$.
The following conjecture was proposed by Pil\'{s}niak and Wo\'{z}niak in~\cite{PW},
where it was also verified for a few classical graph families, including, {\it e.g.},
complete graphs, bipartite graphs and
graphs with maximum degree at most three.

\begin{conjecture}[\cite{PW}]\label{MoMa}
For every graph $G$, $\chi''_{\Sigma}(G) \leq \Delta(G) + 3$.
\end{conjecture}
Note that such postulated upper bound exceeds just by one the bound implied by the famous Total Colouring Conjecture
posed by Vizing~\cite{Vizing2} in 1968 and independently by Behzad~\cite{Behzad} in 1965, and
concerning proper total colourings (without our additional requirement on sum distinction between adjacent vertices).
The both conjectures seem to be very challenging and are in general widely open. The best general result concerning the latter one~\cite{MolloyReedTotal} confirms however the Total Colouring Conjecture up to a (large) additive constant.

The best general upper bound concerning Conjecture~\ref{MoMa} implies that $\chi''_{\Sigma}(G)\leq (1+o(1))\Delta$ for every graph $G$ with maximum degree $\Delta$, see~\cite{Przybylo_asym_optim_total} and~\cite{Przybylo_asymptotic_note}. See also~\cite{total_sum_planar,LiLiuWang_sum_total_K_4,PW,Przybylo_CN_3} for partial results on this conjecture.
In particular Ding et al. first confirmed Conjecture~\ref{MoMa} for planar graphs with sufficiently large maximum degree:
\begin{theorem} [\cite{total_sum_planar}]
Any planar graph $G$ with $\Delta(G) \geq 13$ satisfies $\chi''_\Sigma(G) \leq \Delta(G)+3$.
\end{theorem}
This was then improved by Yang et al. in the following form.
\begin{theorem} [\cite{total_sum_planar2}]
Any planar graph $G$ satisfies $\chi''_\Sigma(G) \leq \max\{\Delta(G)+2,13\}$.
\end{theorem}

Let ${\rm mad}(G)=\max\left\{\frac{2|E(H)|}{|V(H)|},\;H \subseteq G\right\}$ be the \emph{maximum average degree} of a graph $G$, where $V(H)$ and $E(H)$ are the sets of vertices and edges of $H$, respectively.
This is a conventional measure of sparseness of an arbitrary graph (not necessary planar). For more details on this invariant see {\it e.g.} \cite{Coh10,Toft}.
Dong and Wang first made the link between the maximum average degree and the total neighbour sum distinguishing index.
They proved the following result.

\begin{theorem}[\cite{DongWang_mad}] \label{th:summadCite1} %\cite{DongWang_mad}
Any graph $G$ with $\Delta(G) \geq 5$ and ${\rm mad}(G)<3$ satisfies $\chi''_\Sigma(G) \leq \Delta(G)+2$.
\end{theorem}

This subject was intensively studied afterwards, and the following improvement has been announced recently (but no proof of such a supposed fact
was published thus far).
\begin{theorem}[\cite{Qiu_Wang_Liu_Xu}] \label{th:summadCite2}
Any graph $G$ with $\Delta(G) \geq 8$ and ${\rm mad}(G)< \frac92$ satisfies $\chi''_\Sigma(G) \leq \Delta(G)+3$.
\end{theorem}

In this paper, we prove a stronger statement than
the above:

\begin{theorem}\label{th:summad}
Any graph $G$ with $\Delta(G) \geq 8$ and ${\rm mad}(G)<\frac{14}{3}$ satisfies $\chi''_\Sigma(G) \leq \Delta(G)+3$.
\end{theorem}

Recall that the girth of a graph is the length of a shortest cycle in it.
As every planar graph with girth $g$ satisfies ${\rm mad}(G) < \frac{2g}{g-2}$, the following corollary can be easily derived from Theorem~\ref{th:summad}:

\begin{corollary}
Any triangle-free planar graph $G$ with $\Delta(G) \geq 8$ satisfies $\chi''_\Sigma(G) \leq \Delta(G)+3$.
\end{corollary}

\section{Proof of Theorem~\ref{th:summad}}

\subsection{Preliminaries}

Fix an integer $k\geq 8$.
In the following, $n_i(G)$ denotes the number of vertices of degree $i$ in a graph $G$ (and similarly for $n_{i^+}(G)$ with ``at least $i$'' and for $n_{i^-}(G)$ with ``at most $i$'').
We say a graph $G$ is \emph{smaller} than a graph $H$, $G \prec H$ if $|E(G)|+|V(G)|<|E(H)|+|V(H)|$.
 We say a graph is \emph{minimal} for a property when no smaller graph verifies it.
We shall also call any vertex of degree $d$ ($\geq d$, $\leq d$) in a given graph a \emph{$d$-vertex} (\emph{$d^+$-vertex}, \emph{$d^-$-vertex}, resp.) of this graph.
The same nomenclature shall be used for neighbours as well.\\

\subsection{Structural properties of H}

Suppose that $H$ is a minimal graph with maximum degree $\Delta\leq k$,
${\rm mad}(H)<\frac{14}{3}$ and $\chi''_\Sigma(H)>k+3$ (hence $H$ is connected and $\delta(H)\geq 1$).
In the remaining part of the paper we argument that in fact $H$ cannot exist, {\it i.e.} that there exists a tnsd  $(k+3)$-colouring of $H$,  and thus prove Theorem~\ref{th:summad}.

In this subsection we exhibit some structural properties of $H$.

The following lemma shall be very useful to this end.
Its proof was inspired by the research from~\cite{BP14}.
The same result but in the case of lists of integers can also be derived from~\cite{Alon}.

\begin{lemma}[\cite{HocquardPrzybylo1}]\label{MartheLemma}
For any finite sets $L_1,\ldots,L_t$ of real numbers with $|L_i|\geq t$ for $i=1,\ldots,t$,
the set $\{x_1+\ldots+x_t:x_1\in L_1,\ldots,x_t\in L_t; x_i\neq x_j~{\rm for}~i\neq j\}$ contains
at least $\sum_{i=1}^t|L_i|-t^2+1$ distinct elements.
\end{lemma}

\begin{observation}\label{obs3v}
Every $3^-$-vertex $v$ in $H$ can be recoloured (or coloured if it has no colour assigned) so that it has a different colour than its adjacent vertices and incident edges, and so that $v$ is not in conflict with any of its neighbours.
\end{observation}

\begin{proof}
This follows directly by the fact we have $k+3\geq 11$ colours available, while at most $9$ of them might be blocked by the requirements from the thesis.
\end{proof}

\begin{lemma}\label{lemma4v}
For every vertex $v \in V(H)$, % with $d(v) \ge 7$,
$n_{4^+}(v) \ge n_{2^-}(v)+1+ n_{3^-}(v)\times(k - d(v))$.
\end{lemma}

\begin{proof}
Suppose on the contrary that $v\in V(H)$ with $d(v)=d\geq 1$ is adjacent to $\alpha$ $2^-$-vertices $u_1,\dots,u_\alpha$, $\beta$ $3$-vertices $w_1,\dots,w_\beta$ and to $\gamma$ $4^+$-vertices where $\gamma < \alpha+1+(\alpha+\beta)(k-d)$, hence $\alpha+\beta\geq 1$. %$d=d(v)=\alpha+\beta+\gamma$.
Colour $H'=H-\{vu_1,\ldots,vu_\alpha,vw_1,\ldots,vw_\beta\}$ by minimality ({\it i.e.} fix any tnsd  $(k+3)$-colouring of $H'$, which must exist due to the fact that $H$ is our minimal counterexample, while $H' \prec H$, $\Delta(H')\leq k$ and ${\rm mad}(H')<\frac{14}{3}$) and uncolour $u_i$ and $w_j$ for all $i \in \{1,\ldots,\alpha\}$ and $j \in \{1,\ldots,\beta\}$. Let $L_i$ and $L'_j$, with $i \in \{1,\ldots,\alpha\}$ and $j \in \{1,\ldots,\beta\}$ be the sets of available colours respectively for the edges $vu_i$ and $vw_j$ ({\it i.e.} those colours in $\{1,\ldots,k+3\}$ not used by their adjacent edges in $H$ or $v$).
Note that
$|L_1|,\dots,|L_\alpha| \ge \alpha + \beta + 1 + k -d$ and $|L'_1|,\dots,|L'_\beta| \ge \alpha + \beta+ k -d$.
By Lemma~\ref{MartheLemma}, we may extend this colouring to a (partial) proper colouring of $H$ in different ways, obtaining at least $\alpha(\alpha+\beta)+\alpha+\alpha(k-d)+\beta(\alpha+\beta)+\beta(k-d)-(\alpha+\beta)^2+1=\alpha+1+(\alpha+\beta)(k-d)>\gamma$ distinct sums for~$v$.
Thus we can do it in such a way that $v$ is no in conflict with any of its $4^+$-neighbours.
By Observation~\ref{obs3v} we therefore obtain a contradiction. $\blacksquare$\\

\end{proof}

\begin{corollary}\label{obs2v}
For every vertex $v \in V(H)$ with $d(v) \ge 7$, $n_{2^-}(v) \le d(v)-5$.
\end{corollary}

\begin{proof}
Suppose to the contrary that $n_{2^-}(v) \ge d(v)-4$ for some $v \in V(H)$ with $d(v) \ge 7$. By Lemma~\ref{lemma4v} we have: $n_{4^+}(v) \ge n_{2^-}(v)+1+ n_{3^-}(v)\times(k - d(v))\ge d(v)-4+1+n_{3^-}(v)\times(k - d(v)) \ge d(v)-3$.

Consequently, we have:
$d(v) \ge n_{2^-}(v)+n_{4^+}(v) \ge d(v)-4+d(v)-3=2d(v)-7$.
Hence, $d(v) \le 7$, {\it i.e.} $d(v)=7$. As $k \ge 8$, then by Lemma~\ref{lemma4v} we thus obtain:
$n_{4^+}(v) \ge n_{2^-}(v)+1+ n_{3^-}(v)\times(k - d(v))\ge 4 + n_{3^-}(v)\times(k - d(v)) \ge d(v)$, a contradiction with the fact that $n_{2^-}(v) \ge d(v)-4$. $\blacksquare$
\end{proof}

Within the proof of the remaining structural properties of $H$, aggregated in Claim~\ref{claimstructure} below, we shall apply several times the following algebraic tool due to Alon~\cite{Alon}.

\begin{theorem}[Combinatorial
Nullstellensatz]\label{Comb_Nul} Let $\mathbb{F}$ be an arbitrary
field, and let $P=P(x_1,\ldots,x_n)$ be a polynomial in
$\mathbb{F}[x_1,\ldots,x_n]$.
Suppose the coefficient of a monomial $x_1^{k_1}\ldots
x_n^{k_n}$, where each $k_i$ is a non-negative
integer, is non-zero in $P$ and the degree ${\rm deg}(P)$ of
$P$ equals $\sum_{i=1}^n k_i$.
If moreover $S_1,\ldots,S_n$ are any
subsets of $\mathbb{F}$ with $|S_i|>k_i$ for $i=1,\ldots,n$,
then there are $s_1\in
S_1,\ldots,s_n\in S_n$ so that $P(s_1,\ldots,s_n)\neq 0$.
\end{theorem}

\begin{claim}\label{claimstructure}
The graph $H$ does not contain any of:

\begin{enumerate}
\item[(C1)] \label{c1} a $2^-$-vertex $v$ adjacent to a $(\frac{k}{2}+1)^-$-vertex $u$;
\item[(C2)] \label{c2} a $4^-$-vertex $v$ adjacent to a $4^-$-vertex $u$;
\item[(C3)] \label{c3} a $3^-$-vertex $v$ adjacent to a $5^-$-vertex $u$;
\item[(C4)] \label{c4} a $5$-vertex $v$ adjacent to three $4$-vertices $v_1,v_2,v_3$;
\item[(C5)] \label{c5} a $6$-vertex $v$ adjacent to a $3^-$-vertex $u$ and to a $4^-$-vertex $w$;
\item[(C6)] \label{c12} a $7$-vertex $v$ adjacent to a $2^-$-vertex $u$, to a $3^-$-vertex $w$ and to a $4^-$-vertex $y$.
\item[(C7)] \label{c6} a vertex $v$ of degree $d\geq 8$ adjacent to $(d-7)$ $2^-$-vertices $v_1,\ldots,v_{d-7}$, to two $3^-$-vertices $u_1,u_2$ and to a $4^-$-vertex $w$;
\item[(C8)] \label{c11} a vertex $v$ of degree $d=\Delta\geq 3$ adjacent to $(d-2)$ $3^-$-vertices $v_1,\ldots,v_{d-2}$ and to one $4^-$-vertex $u$.
\end{enumerate}
\end{claim}
\medskip

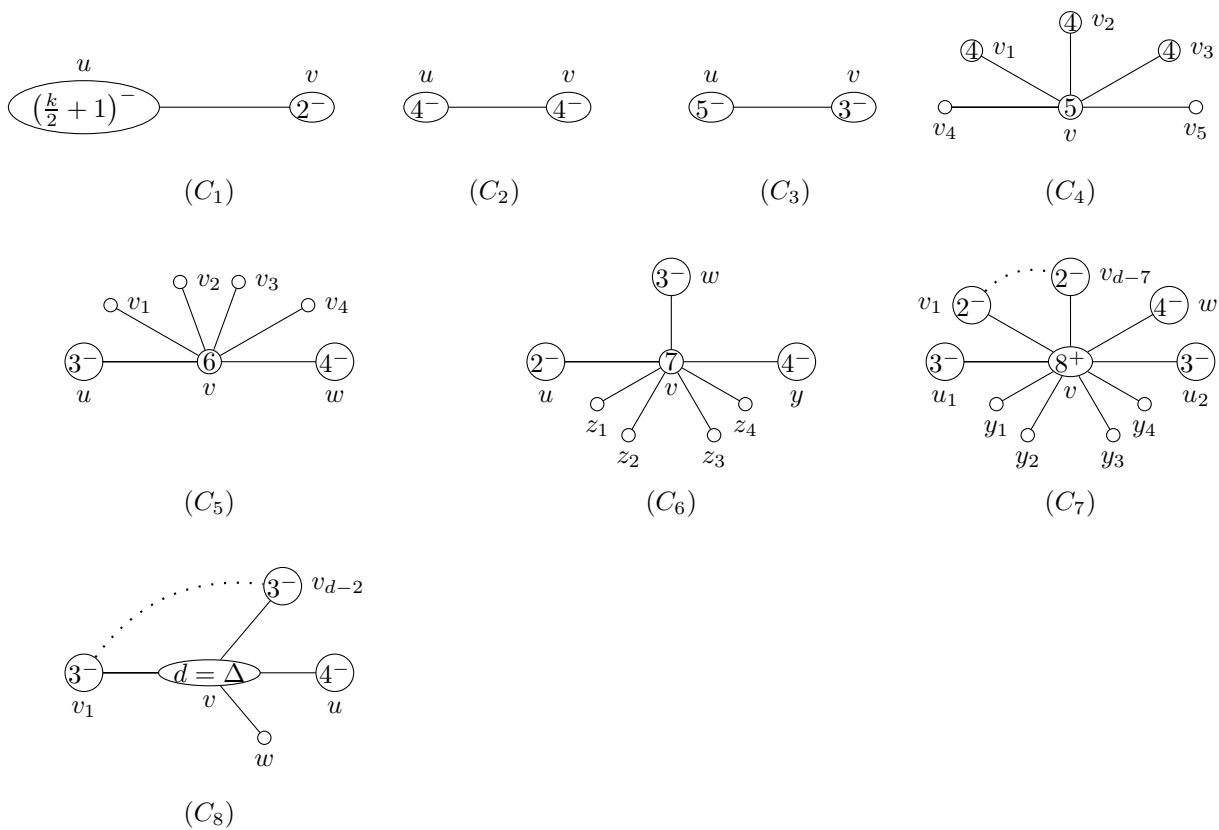
\begin{figure}[!h]
\begin{center}
\centering
\begin{tikzpicture}[scale=0.75,auto]
	\tikzstyle{w}=[draw,circle,fill=white,minimum size=5pt,inner sep=0pt]
	\tikzstyle{b}=[draw,circle,fill=black,minimum size=5pt,inner sep=0pt]
	\tikzstyle{t}=[rectangle,minimum size=5pt,inner sep=0pt]

	 \tikzstyle{whitenode}=[draw,ellipse,fill=white,minimum size=9pt,inner sep=0pt]
	 \tikzstyle{blacknode}=[draw,circle,fill=black,minimum size=9pt,inner sep=0pt]
	 \tikzstyle{texte}=[minimum size=9pt,inner sep=0pt]

\draw (0,0) node[whitenode] (u) [label=above:$u$] {$\left(\frac{k}{2}+1\right)^-$}
%--++ (0:1cm) node[b] (v1) [label=above:$v$] {}
--++ (0:4cm) node[whitenode] (b) [label=above:$v$] {$2^-$};

\draw (2.2,-1.5) node[t] (t1) {$(C_1)$};

\draw (6,0) node[whitenode] (b) [label=above:$u$] {$4^-$}
--++ (0:2.5cm) node[whitenode] (b) [label=above:$v$] {$4^-$};

\draw (7.2,-1.5) node[t] (t1) {$(C_2)$};

\draw (11,0) node[whitenode] (b) [label=above:$u$] {$5^-$}
--++ (0:2.5cm) node[whitenode] (b) [label=above:$v$] {$3^-$};

\draw (12.3,-1.5) node[t] (t1) {$(C_3)$};

\draw (17.3,0) node[whitenode] (v) [label=below:$v$] {$5$}
--++ (0:-2.2cm) node[w] (u)  [label=below:$v_4$] {}
--++ (0:4.4cm) node[w] (w) [label=below:$v_5$] {};

\draw (v)
--++ (150:2cm) node[w] [label=right:$v_1$] (v5) {$4$};

\draw (v)
--++ (90:1.5cm) node[w] [label=right:$v_2$] (v6)  {$4$};

\draw (v)
--++ (30:2cm) node[w] [label=right:$v_3$] (v7)  {$4$};

\draw (17.3,-1.5) node[t] (t1) {$(C_4)$};

\draw (2.2,-4.5) node[whitenode] (v) [label=below:$v$] {$6$}
--++ (0:-2.2cm) node[w] (u)  [label=below:$u$] {$3^-$}
--++ (0:4.4cm) node[w] (w) [label=below:$w$] {$4^-$};

\draw (v)
--++ (150:2cm) node[w] [label=right:$v_1$] (v5) {};

\draw (v)
--++ (110:1.5cm) node[w] [label=right:$v_2$] (v6)  {};

\draw (v)
--++ (70:1.5cm) node[w] [label=right:$v_3$] (v7)  {};

\draw (v)
--++ (30:2cm) node[w] [label=right:$v_4$] (v8)  {};

\draw (2.2,-7) node[t] (t1) {$(C_5)$};

\draw (10.3,-4.5) node[whitenode] (v) [label=below:$v$] {$7$}
--++ (0:-2.2cm) node[w] (u)  [label=below:$u$] {$2^-$}
--++ (0:4.4cm) node[w] (w) [label=below:$y$] {$4^-$};

\draw (v)
--++ (90:1.5cm) node[w] [label=right:$w$] (v6)  {$3^-$};

\draw (v)
--++ (-30:1.5cm) node[w] (v8) [label=below:$z_4$] {};

\draw (v)
--++ (-60:1.5cm) node[w] (v9) [label=below:$z_3$] {};

\draw (v)
--++ (-120:1.5cm) node[w] (v10) [label=below:$z_2$] {};

\draw (v)
--++ (-150:1.5cm) node[w] (v11) [label=below:$z_1$]  {};

\draw (10.3,-7) node[t] (t1) {$(C_6)$};

\draw (17.3,-4.5) node[whitenode] (v) [label=below:$v$] {$8^+$}
--++ (0:-2.2cm) node[w] (u)  [label=below:$u_1$] {$3^-$}
--++ (0:4.4cm) node[w] (w) [label=below:$u_2$] {$3^-$};

\draw (v)
--++ (150:2cm) node[w] [label=left:$v_1$] (v5) {$2^-$};

\draw (v)
--++ (90:1.5cm) node[w] [label=right:$v_{d-7}$] (v6)  {$2^-$};

\draw (v)
--++ (30:2cm) node[w] [label=right:$w$] (v7)  {$4^-$};

\draw (v)
--++ (-30:1.5cm) node[w] (v8) [label=below:$y_4$] {};

\draw (v)
--++ (-60:1.5cm) node[w] (v9) [label=below:$y_3$] {};

\draw (v)
--++ (-120:1.5cm) node[w] (v10) [label=below:$y_2$] {};

\draw (v)
--++ (-150:1.5cm) node[w] (v11) [label=below:$y_1$] {};

\draw (v5) edge [bend left,loosely dotted,thick] node {} (v6);

\draw (17.3,-7) node[t] (t1) {$(C_7)$};

\draw (2.2,-10) node[whitenode] (v) [label=below:$v$] {$d=\Delta$}
--++ (0:-2.2cm) node[w] (u)  [label=below:$v_1$] {$3^-$}
--++ (0:4.4cm) node[w] (w) [label=below:$u$] {$4^-$};

\draw (v)
--++ (50:2cm) node[w] [label=right:$v_{d-2}$] (v7)  {$3^-$};

\draw (u) edge [bend left,loosely dotted,thick] node {} (v7);

\draw (v)
--++ (-50:1.5cm) node[w] (v8) [label=below:$w$] {};

\draw (2.2,-12.5) node[t] (t1) {$(C_8)$};

\end{tikzpicture}
\caption{Forbidden configurations in $H$%(where solid vertices have degrees as presented in the figure, hollow vertices may have additional edges and may coincide with other vertices).
}\label{FCG_fig}
\end{center}
\end{figure}

\begin{proof}
We shall argument `reducibility' of each of these 8 configurations separately,
following a similar pattern of reasoning.
I.e., we shall first suppose by contradiction that a given configuration exists in $H$.
Then we shall consider a graph $H'$ smaller than $H$ with $\Delta(H')\leq k$ and ${\rm mad}(H')<\frac{14}{3}$
(usually guaranteing these properties by constructing $H'$ simply via deleting some edges or vertices from $H$),
and \emph{colour} it \emph{by minimality},
what shall mean from now on that we choose any tnsd $(k+3)$-colouring for $H'$.
Finally, in each case, we shall obtain a contradiction by extending the colouring chosen to a tnsd $(k+3)$-coloring of the entire $H$.
Whenever we analyze a partial colouring of a graph, the \emph{sum at a} given \emph{vertex}, $s(v)$ is defined as above, but every uncoloured edge and vertex contributes $0$ to this sum. We write that $u$ and $v$ are \emph{sum-distinguished}, if $s(u)\neq s(v)$.

\begin{enumerate}
\item[1.] Suppose there exists a $2^-$-vertex $v$ adjacent to a $(\frac{k}{2}+1)^-$-vertex $u$ in $H$. Colour $H'=H-\{uv\}$ by minimality and uncolour $v$. In order to colour $uv$ so that a (partial) $(k+3)$-colouring of $H$ obtained is proper we have to avoid at most $\frac{k}{2}+2$ colours,
     and possibly at most $\frac{k}2$ more colours to ensure the sum-distinction (of $u$ from its neighbours other than $v$). Hence, we have at least one colour left to extend the colouring, and thus obtain
     a tnsd $(k+3)$-colouring of $H$ via Observation~\ref{obs3v} applied to $v$, a contradiction.
\item[2;3.] Suppose there exists an edge $uv$ with $d(u),d(v)\leq 4$ or with $d(u)\leq 5$ and $d(v)\leq 3$  in $H$. Denote the neighbours of $u$ other than $v$ by $u_1,\ldots,u_q$, and denote the neighbours of $v$ other than $u$ by $v_1,\ldots,v_p$ (hence $d(u)=q+1$ and $d(v)=p+1$).
    By the minimality of $H$ there exists a tnsd $(k+3)$-colouring $c$ of $H'=H-\{uv\}$.
    Let us now undelete the edge $uv$ and remove the colours from $u$ and $v$.
    In order to extend the current partial colouring of $H$ to its proper total $(k+3)$-colouring we may use at least $11-6=5$ colours for $u$, at least $11-6=5$ colours for $uv$ and at least $11-6=5$ colours for $v$ if $d(u),d(v)\leq 4$, or otherwise: at least $3$ colours for $u$, at least $5$ colours for $uv$ and at least $7$ colours for $v$. Denote the respective lists of available colours by $L_u,L_{uv},L_v$, the sum at $u_i$
    by $s_i$ and the sum at $v_j$
    by $s'_j$ for $i=1,\ldots,q$, $j=1,\ldots,p$. Consider a polynomial with real variables:
    \begin{eqnarray}
    f(x_0,x_1,x_2) &=& (x_0-x_1)(x_0-x_2)(x_1-x_2)\left(x_0+\sum_{i=1}^q c(uu_i)-x_2-\sum_{i=1}^p c(vv_i)\right)\nonumber\\
    &\times&\prod_{i=1}^q\left(x_0+x_1+\sum_{j=1}^qc(uu_j)-s_i\right)\prod_{i=1}^p\left(x_2+x_1+\sum_{j=1}^pc(vv_j)-s'_i\right).\nonumber
    \end{eqnarray}
    Note that in order to extend the colouring $c$ to a tnsd $(k+3)$-colouring %$c$
    of $H$ it is now sufficient to find a non-zero (i.e. with non-zero value of $f$) substitution for $f$ such that $x_0\in L_u$, $x_1\in L_{uv}$ and $x_2\in L_v$. It is thus the more sufficient to find a non-zero substitution from these list for the polynomial $g$ defined as $g(x_0,x_1,x_2):=f(x_0,x_1,x_2)\cdot(x_0+x_1)^{3-q}(x_2+x_1)^{3-p}$ if $d(u),d(v)\leq 4$ or by $g(x_0,x_1,x_2):=f(x_0,x_1,x_2)\cdot(x_0+x_1)^{4-q}(x_2+x_1)^{2-p}$ otherwise.
    In the first of these cases however, the coefficient of the monomial $x_0^4x_1^3x_2^3$ in $g$ is the same as in $$h_1(x_0,x_1,x_2)=(x_0-x_1)(x_0-x_2)^2(x_1-x_2)(x_0+x_1)^{3}(x_2+x_1)^{3},$$ and equals\footnote{This and further computations %calculations
    were obtained %done
    by means of a computer program; one may verify these using {\it e.g.} % can be easily verified with aid of a computer, e.g. using Mathematica...
    \emph{Wolfram Mathematica}.}: $2$. Analogously, in the second case, the coefficient of the monomial $x_0^2x_1^3x_2^5$ in $g$ is the same as in $$h_2(x_0,x_1,x_2)=(x_0-x_1)(x_0-x_2)^2(x_1-x_2)(x_0+x_1)^{4}(x_2+x_1)^{2},$$ and equals also: $2$.
    In the both cases we thus obtain a contradiction by the Combinatorial Nullstellensatz.
\item[4.] Suppose there exists a $5$-vertex $v$ with $N(v)=\{v_1,\ldots,v_5\}$ such that $d(v_1)=d(v_2)=d(v_3)=4$ in $H$.
    By the minimality of $H$ there exists a tnsd $(k+3)$-colouring $c$ of $H'=H-\{vv_1,vv_2,vv_3\}$.
    Delete the colours of $v,v_1,v_2,v_3$. We associate variables $x_0,x_1,x_2,x_3,x_4,x_5,x_6$ with $v,vv_1,vv_2,vv_3,v_1,v_2,v_3$, respectively.
    For these we
    denote the lists of their available colours by $L_0,\ldots,L_6$ (obtained after excluding from $\{1,\ldots,k+3\}$ the colours already used on their respective adjacent or incident vertices and edges),
    respectively. Then $|L_0|\geq 7, |L_1|,|L_2|,|L_3|\geq 6, |L_4|,|L_5|,|L_6|\geq 5$.
    Let
    \begin{eqnarray}
    f(x_0,\ldots,x_6) &=& (x_0-x_1)(x_0-x_2)(x_0-x_3)(x_0-x_4)(x_0-x_5)(x_0-x_6)(x_1-x_2)(x_1-x_3)(x_2-x_3)\nonumber\\
                      &\times& (x_1-x_4)(x_2-x_5)(x_3-x_6)\prod_{i=1}^3(x_0+x_1+x_2+x_3+s(v)-x_i-x_{i+3}-s(v_i))\nonumber\\
                      &\times& \prod_{i=4}^5(x_0+x_1+x_2+x_3+s(v)-s(v_i)) \prod_{i=1}^3\prod_{u\in N(v_i)\smallsetminus\{v\}}(x_i+x_{i+3}+s(v_i)-s(u))\nonumber
    \end{eqnarray}
    (where $s(w)$ refers to the contemporary partial sum for every vertex $w$ in $H$). Note that the coefficient of the monomial $x_0^6x_1^5x_2x_3^5x_4^2x_5^3x_6^4$ in $f$ is the same as in the following polynomial:
    \begin{eqnarray}
    g(x_0,\ldots,x_6) &=& (x_0-x_1)(x_0-x_2)(x_0-x_3)(x_0-x_4)(x_0-x_5)(x_0-x_6)(x_1-x_2)(x_1-x_3)(x_2-x_3)\nonumber\\
                      &\times& (x_1-x_4)(x_2-x_5)(x_3-x_6)(x_0+x_2+x_3-x_4)(x_0+x_1+x_3-x_5)(x_0+x_1+x_2-x_6)\nonumber\\
                      &\times& (x_0+x_1+x_2+x_3)^2(x_1+x_4)^3(x_2+x_5)^3(x_3+x_6)^3,\nonumber
    \end{eqnarray}
    and equals $16$. By the Combinatorial Nullstellensatz we thus may extend our colouring to a tnsd $(k+3)$-colouring of $H$, a contradiction.
\item[5.] Suppose there exists a $6$-vertex $v$ adjacent to a $3^-$-vertex $u$, to a $4^-$-vertex $w$ and to vertices $v_1,v_2,v_3,v_4$ in $H$.
    By the minimality of $H$ there exists a tnsd $(k+3)$-colouring $c$ of $H'=H-\{vu,vw\}$.
    Delete the colours of $u,v,w$ and associate variables $x_0,x_1,x_2,x_3,x_4$ with $v,vu,vw,u,w$, respectively.
    Denote the lists of available colours for these by $L_0,\ldots,L_4$, resp., and note that $|L_0|\geq 3, |L_1|\geq 5, |L_2|\geq 4, |L_3|\geq 7, |L_4|\geq 5$.
    Consider a polynomial:
    \begin{eqnarray}
    f(x_0,\ldots,x_4) &=& (x_0-x_1)(x_0-x_2)(x_0-x_3)(x_0-x_4)(x_1-x_2)(x_1-x_3)(x_2-x_4)\nonumber\\
                      &\times& (x_0+x_2+s(v)-x_3-s(u))(x_0+x_1+s(v)-x_4-s(w))\nonumber\\
                      &\times& \prod_{i=1}^4(x_0+x_1+x_2+s(v)-s(v_i))\nonumber\\
                      &\times& \prod_{y\in N(u)\smallsetminus\{v\}}(x_1+x_3+s(u)-s(y))
                      \prod_{y\in N(w)\smallsetminus\{v\}}(x_2+x_4+s(w)-s(y))\nonumber
    \end{eqnarray}
    and set $g(x_0,\ldots,x_4)=f(x_0,\ldots,x_4)\cdot(x_1+x_3)^{3-d(u)}(x_2+x_4)^{4-d(w)}$.
    Note that the coefficient of the monomial $x_0^2x_1^4x_2^3x_3^5x_4^4$ in $g$ is the same as in:
    \begin{eqnarray}
    h(x_0,\ldots,x_4) &=& (x_0-x_1)(x_0-x_2)(x_0-x_3)(x_0-x_4)(x_1-x_2)(x_1-x_3)(x_2-x_4)\nonumber\\
                      &\times& (x_0+x_2-x_3)(x_0+x_1-x_4)(x_0+x_1+x_2)^4(x_1+x_3)^2(x_2+x_4)^3,\nonumber
    \end{eqnarray}
    and equals $-10$. By the Combinatorial Nullstellensatz there exists a non-zero substitution for $g$, hence the more for $f$, from the corresponding lists $L_0,\ldots,L_4$, and thus we may extend our partial colouring to a tnsd $(k+3)$-colouring of $H$, a contradiction.
\item[6.] Suppose there exists a $7$-vertex $v$ adjacent to a $2^-$-vertex $u$, to a $3^-$-vertex $w$ and to a $4^-$-vertex $y$ in $H$.
    Denote the remaining neighbours of $v$ by $z_1,z_2,z_3,z_4$.
    By the minimality of $H$ there exists a tnsd $(k+3)$-colouring $c$ of $H'=H-\{vu,vw,vy\}$.
    Delete the colours of $u,w,y$ and associate variables $x_1,x_2,x_3,x_4$ with $vu,vw,vy,y$, respectively.
    Denote the lists of available colours for these by $L_1,L_2,L_3,L_4$, resp., and note that $|L_1|\geq 5, |L_2|\geq 4, |L_3|\geq 3, |L_4|\geq 4$.
    Consider a polynomial (and note that by Observation~\ref{obs3v} we shall be able to colour properly vertices $u$ and $w$ at the end so that these are sum distinguished from their neighbours, thus we omit the corresponding requirements within the polynomial below):
    \begin{eqnarray}
    f(x_1,x_2,x_3,x_4) &=& (x_1-x_2)(x_1-x_3)(x_2-x_3)(x_3-x_4)\prod_{i=1}^4(x_1+x_2+x_3+s(v)-s(z_i))\nonumber\\
                      &\times& (x_1+x_2+s(v)-x_4-s(y))
                      \prod_{z\in N(y)\smallsetminus\{v\}}(x_3+x_4+s(y)-s(z)).\nonumber
    \end{eqnarray}
    Let $g(x_1,x_2,x_3,x_4)=f(x_1,x_2,x_3,x_4)\cdot (x_3+x_4)^{4-d(y)}$.
    Note that the coefficient of the monomial $x_1^4x_2^3x_3^2x_4^3$ in $g$ is the same as in:
    $$h(x_1,x_2,x_3,x_4) = (x_1-x_2)(x_1-x_3)(x_2-x_3)(x_3-x_4)(x_1+x_2+x_3)^4(x_1+x_2-x_4)(x_3+x_4)^3,$$
    and equals $-6$. Therefore, analogously as above we may extend our colouring, first to $vu,vw,vy,y$ by the Combinatorial Nullstellensatz, and then to $u$ and $w$ by Observation~\ref{obs3v}, to a tnsd $(k+3)$-colouring of $H$, a contradiction.
\item[7.] Suppose there exists
    a vertex $v$ of degree $d\geq 8$
    adjacent to $(d-7)$ $2^-$-vertices $v_1,\ldots,v_{d-7}$, to two $3^-$-vertices $u_1,u_2$ and to a $4^-$-vertex $w$ in $H$. The remaining neighbours of $v$ we denote by $y_1,y_2,y_3,y_4$.
    By the minimality of $H$ there exists a tnsd $(k+3)$-colouring $c$ of $H'=H-\{vv_1,vu_1,vu_2,vw\}$.
    Delete the colours of $v,v_1,\ldots,v_{d-7},u_1,u_2,w$ and associate variables $x_0,x_1,x_2,x_3,x_4,x_5$ with $v,vv_1,vu_1,vu_2,vw,w$, respectively.
    Denote the lists of available colours for these by $L_0,\ldots,L_5$, resp., and note that $|L_0|\geq 3, |L_1|\geq 6, |L_2|,|L_3|\geq 5, |L_4|\geq 4, |L_5|\geq 5$.
    Consider a polynomial
    (the vertices $v_1,\ldots,v_{d-7},u_1,u_2$ shall be coloured at the end via Observation~\ref{obs3v}):
    \begin{eqnarray}
    f(x_0,\ldots,x_5) &=& (x_0-x_1)(x_0-x_2)(x_0-x_3)(x_0-x_4)(x_0-x_5)(x_1-x_2)(x_1-x_3)(x_1-x_4)\nonumber\\
                      &\times& (x_2-x_3)(x_2-x_4)(x_3-x_4)(x_4-x_5)\prod_{i=1}^4(x_0+x_1+x_2+x_3+x_4+s(v)-s(y_i))\nonumber\\
                      &\times& (x_0+x_1+x_2+x_3+s(v)-x_5-s(w))\prod_{z\in N(w)\smallsetminus\{w\}}(x_4+x_5+s(w)-s(z)).\nonumber
    \end{eqnarray}
    Let $g(x_0,\ldots,x_5)=f(x_0,\ldots,x_5)\cdot(x_4+x_5)^{4-d(w)}$. Then the coefficient of the monomial $x_0x_1^5x_2^4x_3^3x_4^3x_5^4$ in $g$ is the same as in:
    \begin{eqnarray}
    h(x_0,\ldots,x_5) &=& (x_0-x_1)(x_0-x_2)(x_0-x_3)(x_0-x_4)(x_0-x_5)(x_1-x_2)(x_1-x_3)(x_1-x_4)\nonumber\\
                      &\times& (x_2-x_3)(x_2-x_4)(x_3-x_4)(x_4-x_5)(x_0+x_1+x_2+x_3+x_4)^4\nonumber\\
                      &\times& (x_0+x_1+x_2+x_3-x_5)(x_4+x_5)^3.\nonumber
    \end{eqnarray}
    and equals $5$. By the Combinatorial Nullstellensatz and Observation~\ref{obs3v} we may thus extend our partial colouring to a tnsd $(k+3)$-colouring of $H$, a contradiction.
\item[8.] Suppose $v$ is a vertex of degree $d=\Delta\geq 3$ adjacent to $(d-2)$ $3^-$-vertices $v_1,\ldots,v_{d-2}$ and to one $4^-$-vertex $u$ in $H$.
    Denote the remaining neighbour of $v$ by $w$.
    By the minimality of $H$ there exists a tnsd $(k+3)$-colouring $c$ of $H'=H-\{vv_1,\ldots,vv_{d-2},vu\}$.
    Delete the colours of $v,v_1,\ldots,v_{d-2},u$. First extend such a partial colouring of $H$ by choosing
    a colour (in $\{1,\ldots,k+3\}$) for $v$ in the following manner. If $d(v_1)=3$, denote the colours associated to the edges incident with $v_1$ and different from $vv_1$ by $a$ and $b$, and if $c(vw)\notin \{a,b\}$, choose for $v$ any colour in $\{a,b\}\smallsetminus\{c(w)\}$.
    In all other cases, choose for $v$ any colour distinct from $c(vw)$ and $c(w)$. Denote the colour of $v$ by $c(v)$.
    Then choose a colour $c(u)$ for $u$ as small as possible
    (and note that as our total colouring must be proper, this implies that either $u$ or some of its incident edges other than $uv$ has now colour at most $5$). Next we choose any colour $c(uv)$ so that the obtained (partial) colouring of $H$ is proper and $u$ is sum distinguished from its neighbours other than $v$ (this is possible, as $k+3>9$). Then we subsequently choose greedily colours for $vv_2,\ldots,vv_{d-2}$ so that the obtained partial total colouring of $H$ is proper. Finally we choose a colour $c(vv_1)$ for $vv_1$ distinct from the colours of its incident edges and the colour of $v$ (by our choice of $c(v)$, this blocks at most $\Delta+1$ choices) so that the sum at $v$ is distinct from the sum at $w$,
    and if $\Delta\leq k-1$, also distinct from the sum at $u$.
    We complete our colouring by choosing the colours for $v_1,\ldots,v_{d-2},u$ consistently with %via
    Observation~\ref{obs3v}. In order to see that the obtained colouring of $H$ is sum distinguishing it is sufficient to note that the sum at $v$ is distinct from the sum at $u$ when $\Delta=k$. Indeed, %if $\Delta=8$, then
    the sum of colours incident with $v$ except for the colour of $uv$ equals at least $1+\ldots+k=k(k+1)/2$, while by our choice of the colour for $u$, the sum of its incident colours except the one of $uv$ is at most $5+(k+3)+(k+2)+(k+1)=3k+11<k(k+1)/2$ (for $k\geq 8$). Thus we obtain a contradiction with the minimality of $H$. $\blacksquare$

\end{enumerate}

\end{proof}

\subsection{Discharging procedure}

In this subsection we use the discharging technique exploiting the vertices of the graph $H$.
For this aim we first define the \emph{weight
function} $\omega: V(H) \rightarrow \mathbb{R}$ by setting $\omega(x)=d(x)-\frac{14}{3}$ for every $x\in V(H)$.
Next we shall apply so called \emph{Ghost vertices method}, introduced earlier by Bonamy, Bousquet and Hocquard~\cite{BonamyEtAl},
and based on the following observation (where given any subsets $U,U'\subseteq V(H)$ and a vertex $v$, $d_U(v)$ denotes the number of neighbours of $v$ from $U$, while $E(U,U')$ is the set of edges joining $U$ and $U'$ in the graph $H$).

\begin{observation}\label{obs1v}
\noindent Let $V_1 \cup V_2$ be a partition of $V(H)$ where, say $V_1$ is the set of vertices of degree at least $3$ and $V_2$ -- the set of vertices of degree at most $2$ in $H$;
\begin{itemize}
\item every vertex $u$ in $H$ has an initial weight $w(u)=d(u)-\frac{14}{3}$.
\item If we can discharge the weights in $H$ so that:
\begin{enumerate}
\item every vertex in $V_1$ has a non-negative weight;
\item and every vertex $u$ in $V_2$ has a final weight of at least $d(u)-\frac{14}{3}+d_{V_1}(u)$, then\\

for the new weight assignment $\omega'$, we have $\sum_{v \in V_2} (d(v)-\frac{14}{3}+d_{V_1}(v)) \leq \sum_{v \in V_2} \omega'(v)$, as well as\\

$\sum_{v \in V} \omega(v)=\sum_{v \in V} \omega'(v)$ and $\sum_{v \in V_1} \omega'(v) \geq 0$. Therefore,

\begin{eqnarray}
   \sum_{v \in V_1} \left(d_{V_1}(v)-\frac{14}{3}\right) & \geq & \sum_{v \in V_1} \left(d_{V_1}(v)-\frac{14}{3}\right) + \sum_{v \in V_2} (d(v)-\frac{14}{3}+d_{V_1}(v)) - \sum_{v \in V_2} \omega'(v) \nonumber\\
    & \geq & \sum_{v \in V_1} \left(d_{V_1}(v)-\frac{14}{3}\right) + |E(V_1,V_2)|+\sum_{v \in V_2} \left(d(v)-\frac{14}{3}\right) - \sum_{v \in V_2} \omega'(v) \nonumber\\
     & \geq & \sum_{v \in V_1} \left(d(v)-\frac{14}{3}\right) + \sum_{v \in V_2} \left(d(v)-\frac{14}{3}\right) - \sum_{v \in V_2} \omega'(v) \nonumber\\
     & \geq & \sum_{v \in V} \omega(v) - \sum_{v \in V_2} \omega'(v) \nonumber\\
   & \geq & \sum_{v \in V_1} \omega'(v) \nonumber\\
    & \geq & 0. \nonumber
\end{eqnarray}
\end{enumerate}

Thus we can conclude that ${\rm mad}(H) \geq {\rm mad}(H[V_1]) \geq \frac{14}{3}$.

\end{itemize}

\end{observation}

In other words, the vertices in $V_2$ can be seen but, in a way, do not contribute to the sum analysis.
\bigskip

In order to finish the proof of Theorem~\ref{th:summad}, it suffices to obtain a contradiction, \emph{e.g.} with the fact that ${\rm mad}(H)<\frac{14}{3}$,
implying that in fact no counterexample to its thesis may exist.
By Observation~\ref{obs1v}, it is thus enough to
redistribute the weight (defined by $\omega$ above) in $H$
so that every vertex of degree at least $3$ has a non-negative resulting weight, every vertex of degree $2$ has weight at least $2-\frac{14}{3}+2$ and every vertex of degree $1$ has weight at least $1-\frac{14}{3}+1$.

\medskip

The discharging rules we shall use for this aim are defined as follows:

\begin{enumerate}
\item[(R1)] A vertex of degree $d\geq 6$ gives $1$ to every adjacent $1$-vertex and to every adjacent $2$-vertex.
\item[(R2)] A vertex of degree $d\geq 6$ gives $\frac59$ to every adjacent $3$-vertex.
\item[(R3)] A vertex of degree $d\geq 5$
gives $\frac{1}{6}$ to every adjacent $4$-vertex.
\end{enumerate}

Let $v$ be a vertex in $H$. We consider different cases depending on the degree of $v$.

\begin{itemize}
\item Assume $d(v)=1$. By $(C1)$, $v$ is adjacent to a vertex of degree at least $6$. Thus, by $(R1)$, $v$ receives $1$. So every vertex of degree $1$ in $H$ has an initial weight of $-\frac{11}{3}$, gives nothing according to our rules and receives $1$, hence has the final weight of $-\frac{8}{3}$.

\item Assume $d(v)=2$. By $(C1)$, $v$ is adjacent to two vertices of degree at least $6$, and thus receives $2 \times 1$ by $(R1)$ and gives away nothing according to the rules above. So every vertex of degree $2$ in $H$ has an initial weight of $-\frac{8}{3}$ and has the final weight of $-\frac{2}{3}$.

\item Assume $d(v)=3$. By $(C3)$, $v$ is adjacent to three vertices of degree at least $6$, and thus receives $3 \times \frac59$ by $(R2)$ and gives away nothing according to the rules above. So every vertex of degree $3$ in $H$ has %a non-negative final weight.
    the final weight $0$.

\item Assume $d(v)=4$. By $(C2)$, $v$ is adjacent to four vertices of degree at least $5$, and thus receives $4 \times \frac16$ by $(R3)$ and gives away nothing according to the rules above. So every vertex of degree $4$ in $H$ has %a non-negative final weight.
    the final weight $0$.

\item Assume $d(v)=5$. By $(C3)$, $v$ is not adjacent to a vertex with degree at most $3$, and by $(C4)$, $v$ is adjacent to at most two vertices of degree $4$, and thus gives at most $2 \times \frac16$ by $(R3)$. Hence, $\omega'(v) \ge 5-\frac{14}{3}-2 \times \frac16=0$. %So every vertex of degree $5$ in $H$ has a non-negative final weight.

\item Assume $d(v)=6$.
Then consider the following subcases:
\begin{itemize}
\item if $v$ is adjacent to a $3^-$-vertex, then by $(C5)$, $v$ is adjacent to five vertices of degree at least $5$. Hence, by $(R1)$ and $(R2)$, $\omega'(v) \geq 6-\frac{14}{3}-1 \times \max \{1 ; \frac59\}=\frac13 \ge 0$.
\item if $v$ is not adjacent to a $3^-$-vertex, then $v$ is adjacent to at most six $4$-vertices. Hence, by $(R3)$, $\omega'(v) \ge 6-\frac{14}{3}-6 \times \frac16=\frac13 \ge 0$.
In both cases $v$ has a non-negative final weight.

\end{itemize}

\item Assume $d(v) \ge 7$. Recall that by Corollary~\ref{obs2v}, $n_{2^-}(v) \le d(v)-5$. Then consider the following subcases:
\begin{itemize}
\item If $n_{2^-}(v) = d(v)-5$ then by $(C6)$ and $(C7)$, $v$ is not adjacent to another $4^-$-vertex. Hence, by $(R1)$, $\omega'(v) \ge d(v)-\frac{14}{3}-(d(v)-5)\times 1 \ge 0$.
\item If $n_{2^-}(v) = d(v)-6$ then either $v$ is adjacent to no $3$-vertex or, by $(C6)$ and $(C7)$, %$v$ is adjacent
to exactly
one $3$-vertex and to no $4$-vertices. Hence, by $(R1)$, $(R2)$ and $(R3)$, $\omega'(v) \ge d(v)-\frac{14}{3}-(d(v)-6)\times 1 - \max\{6\times \frac16;1\times \frac59\} \ge 0$.

\item If $n_{2^-}(v) = d(v)-7$ then $v$ is adjacent to at most three $3$-vertices;
for $d(v)\geq 8$ it follows by $(C7)$, while for $d(v)=7$ by Lemma~\ref{lemma4v}, which implies then that
$n_{4^+}(v) \ge 1+ n_{3^-}(v)(k - d(v))\geq 1+n_{3^-}(v)$.
%The remaining neighbours of $v$ are $4^+$-vertices.
Hence,  by $(R1)$, $(R2)$ and $(R3)$, $\omega'(v) \ge d(v)-\frac{14}{3}-(d(v)-7)\times 1-3\times \frac59 - 4 \times \frac16 \ge 0$.

\item If $n_{2^-}(v) \le d(v)-8$ then by $(C8)$: %and Corollary~\ref{obs2v}:

\begin{itemize}
\item if $v$ is not adjacent to any $4$-vertex then $v$ is adjacent to %at most
$(d(v)-8-\alpha)$ $2^-$-vertices and to at most $(\alpha+6)$ $3$-vertices for some $\alpha \ge 0$.  Hence, by $(R1)$, $(R2)$ and $(R3)$, $\omega'(v) \ge d(v)-\frac{14}{3}-(d(v)-8-\alpha)\times 1-(\alpha+6)\times \frac59=\frac49 \alpha \ge 0$. %because $\alpha \ge 0$.

\item if $v$ is adjacent to at least one $4$-vertex then $v$ is adjacent to %at most
$(d(v)-8-\alpha)$ $2^-$-vertices, %and
to %at most
$(\alpha+6-\beta)$ $3$-vertices and to at most $(\beta+2)$ $4$-vertices for some $\alpha \ge 0$ and some $\beta \ge 1$. Hence, by $(R1)$, $(R2)$ and $(R3)$, $\omega'(v) \ge d(v)-\frac{14}{3}-(d(v)-8-\alpha)\times 1-(\alpha+6-\beta)\times \frac59 - (\beta+2)\times \frac16=\frac49 \alpha + \frac{7}{18}\beta - \frac13 \ge 0$. %because $\alpha \ge 0$ and $\beta \ge 1$.
\end{itemize}

\end{itemize}

\end{itemize}

In all cases $v$ has a non-negative final weight.

This, by Observation~\ref{obs1v}, completes the proof of Theorem~\ref{th:summad}. $\blacksquare$

\end{document}